\documentclass[a4paper,11pt,leqno]{article}
\usepackage{amsmath,amssymb,amsfonts,amsthm,graphicx}
\usepackage[utf8]{inputenc}
\usepackage[T1]{fontenc}
\usepackage{rotating}
\usepackage{mathtools}
\usepackage{thmtools}
\usepackage[colorlinks=true,hypertexnames=false]{hyperref}

\newtheorem{teo}{Theorem}[section]
\newtheorem{lem}[teo]{Lemma}

\newtheorem{cor}[teo]{Corollary}
\newtheorem{ques}[teo]{Question}

\newtheorem*{tea}{Theorem}
\newtheorem{teoA}{Theorem}

\theoremstyle{definition}
\newtheorem{dfn}[teo]{Definition}
\newtheorem{exa}[teo]{Example}
\theoremstyle{plain}

\declaretheoremstyle[
  spaceabove=\topsep, spacebelow=\topsep,
  headfont=\bf,
  notefont=\mdseries, notebraces={(}{)},
  bodyfont=\rmfamily,
  postheadspace=1em,
  qed=$\Diamond$
]{drem}
\declaretheorem[style=drem, name=Remark, numberlike=teo]{rmk}

\newenvironment{acknowledgements}{\par\smallskip\noindent\textit{Acknowledgements.}\enskip}{\par}

\newcommand{\eg}[0]{\emph{e.g.} }
\newcommand{\ie}[0]{\emph{i.e.} }

\newcommand{\ssl}[1]{\underline{#1}}
\newcommand{\srl}[1]{\overline{#1}}

\DeclareFontFamily{T1}{mafra}{}
\DeclareFontShape{T1}{mafra}{m}{n}{<->s*[0.95]yswab}{}
\DeclareFontShape{T1}{mafra}{m}{it}{<->s*[1.0]ygoth}{}
\DeclareSymbolFont{mafrak}{T1}{mafra}{m}{n}
\DeclareSymbolFontAlphabet{\mathfr}{mafrak}

\DeclareSymbolFont{mbbold}{U}{bbold}{m}{n}
\DeclareSymbolFontAlphabet{\mathbbold}{mbbold}

\newcommand{\mr}[1]{\mathrm{#1}}

\newcommand{\mc}[1]{\mathcal{#1}}

\newcommand{\mf}[1]{\mathfr{#1}}

\newcommand{\pgen}[1]{\langle #1 \rangle}

\newcommand{\imp}[0]{\Rightarrow}

\newcommand{\eps}[0]{\varepsilon}

\newcommand{\kk}[0]{\ensuremath{\mathbb{K}}}
\newcommand{\rr}[0]{\ensuremath{\mathbb{R}}}
\newcommand{\zz}[0]{\ensuremath{\mathbb{Z}}}
\newcommand{\nn}[0]{\ensuremath{\mathbb{N}}}

\newcommand{\cc}[0]{\ensuremath{\mathbb{C}}}

\newcommand{\del}[0]{\ensuremath{\partial}}
\newcommand{\dd}[0]{\ensuremath{\mathrm{d}\!}}

\newcommand{\un}[0]{1 \!\! \mathrm{l}}

\newcommand{\Id}[0]{\mathrm{Id}}

\newcommand{\limm}[1]{\textrm{\raisebox{.5ex}{\mbox{$\underset{#1}{\lim}$}}} \:}
\newcommand{\maxx}[1]{\textrm{\raisebox{.5ex}{\mbox{$\underset{#1}{\max}$}}} \:}

\newcommand{\supp}[1]{\textrm{\raisebox{.5ex}{\mbox{$\underset{#1}{\sup}$}}} \:}
\newcommand{\inff}[1]{\textrm{\raisebox{.5ex}{\mbox{$\underset{#1}{\inf}$}}} \:}

\DeclarePairedDelimiter{\nnorm}{\lVert}{\rVert}

\newcommand{\vide}[0]{\varnothing}

\newcommand{\somme}[2]{\overset{#2}{\underset{#1}{\sum}}}
\newcommand{\tosp}[1]{\overset{#1}{\to}}

\newcommand{\cay}[0]{\mathrm{Cay}\!}
\newcommand{\diam}[0]{\mathrm{diam}\!}

\newcommand{\trc}[0]{\mathrm{T}\!\mathrm{C}}
\newcommand{\TC}[0]{\mr{TC}}

\newcommand{\Mm}[0]{\mc{M}}

\newcommand{\FP}[0]{\mathrm{F}\!\mathrm{P}\!}

\newcommand{\BD}[0]{\mathsf{B}\!\mathsf{D}}
\newcommand{\D}[0]{\mathsf{D}}
\newcommand{\LL}[0]{\mathsf{L}}

\newcommand{\maza}[0]{\mu}

\newcommand{\tpoiss}[0]{\cite[Theorem 1.1]{moi-poiss}}

\newcommand{\Gm}{\Gamma}
\newcommand{\gm}{\gamma}

\newcommand{\bsl}\backslash

\begin{document}

\pdfbookmark[0]{Vanishing of lp-cohomology and transportation cost}{article}

\begin{center}
\Large Vanishing of $\ell^p$-cohomology and transportation cost
\vspace*{1cm}

{\large Antoine Gournay}

\smallskip
{\small Universit\'e de Neuch\^atel, Rue \'E.-Argand 11, 2000 Neuch\^atel, Suisse}
\end{center}

\vspace*{.5cm}

\centerline{\textsc{Abstract}}

\begin{center}
\parbox{12cm}{{\small
In this paper, it is shown that the reduced $\ell^p$-cohomology is trivial for a class of finitely generated amenable groups called transport amenable. These groups are those for which there exist a sequence of measures $\xi_n$ converging to a left-invariant mean and such that the transport cost between $\xi_n$ displaced by multiplication on the right by a fixed element and $\xi_n$ is bounded (uniformly in $n$). This class contains groups with controlled F{\o}lner sequence (such as polycylic groups) as well as some wreath products (such as $H' \wr H$ where $H$ is finitely generated Abelian and the ``lamp state'' group $H'$ is finitely generated amenable).
\par\medskip\noindent \emph{Mathematics Subject Classification:} 20J06 (primary), 20E22, 31C05, 43A07, 43A15 (secondary).
}}
\end{center}

\begin{center}
\fbox{\parbox{14cm}{\small
\textbf{Note added in this version (2026).} Examples \ref{ex-reverb} and \ref{ex-reverb2} below are incorrect: a wreath product $H' \wr H$ with $H' \neq \{1\}$ and $H$ infinite is transport amenable only if $H'$ is finite and $H$ is virtually cyclic. Theorem \ref{tvantran-t} and its proof are not affected. An erratum and addendum is appended after the references of this article.
}}
\end{center}

%-------------------------------------------------------------------------
%-------------------------------------------------------------------------
\section{Introduction}\label{s-intro}
%-------------------------------------------------------------------------
%-------------------------------------------------------------------------

\setcounter{teo}{0}
\renewcommand{\theteo}{\thesection.\arabic{teo}}

The study of $\ell^p$-cohomology for discrete spaces has been introduced by Gromov in \cite[\S{}8]{Gro} as an asymptotic invariant of groups (it turns out it is a quasi-isometry invariant, see below). 
The subject matter of this paper is the reduced $\ell^p$-cohomology in degree one of finitely generated groups. Let $\Gamma= (X, E)$ be a graph, the reduced $\ell^p$-cohomology in degree one is the quotient 
\[
\ssl{\ell^p H}^1(\Gamma) := ( \ell^p(E) \cap \nabla \rr^X ) / \srl{\nabla \ell^p(X)}^{\ell^p(E)},
\]
where the gradient of a function on the vertices $g$ is defined by $\nabla g(\gamma,\gamma') = g(\gamma') - g(\gamma)$. See Section \ref{s-def} for details.

Though many definitions and results remain true in a larger setting, this paper deals only with Cayley graphs of a finitely generated group $G$. For a finitely generated group $G$, the reduced $\ell^p$-cohomology of any of its Cayley graphs (for a finite generating set) is isomorphic to the reduced cohomology of the left-regular representation on $\ell^p(G)$, see Puls' paper \cite{Puls-harm} or Martin \& Valette \cite{MV}. Another important result is that $\ell^p$-cohomology is an invariant of quasi-isometry:
\begin{tea}\emph{(see Élek \cite[\S{}3]{El-qi} or Pansu \cite{Pan-qi})}
If two graphs of bounded valency $\Gamma$ and $\Gamma'$ are quasi-isometric, then they have the same $\ell^p$-cohomology (in all degrees, reduced or not).
\end{tea}
The result is actually much more powerful, in the sense that it holds for a large category of measure metric spaces (see above mentioned references). For shorter proofs in more specific situations see Puls \cite[Lemma 6.1]{Puls-pharmbnd} or Bourdon \& Pajot \cite[Théorème 1.1]{BP}. As a consequence, if $G$ is a finitely generated group, the $\ell^p$-cohomology of any two Cayley graphs (for a finite generating set) are isomorphic. Thus, one may speak of the $\ell^p$-cohomology of a group without making reference to a Cayley graph.

The present paper gives a partial answer to a question (dating back at least to Gromov \cite[\S{}8.$A_1$.($A_2$), p.226]{Gro}):
\begin{ques}\label{laquestion}
Let $G$ be an amenable group, is it true that for all $1<p<\infty$, $\ssl{\ell^pH}^1(G)=\{0\}$?
\end{ques}
For $p=2$, the positive answer is a famous result of Cheeger \& Gromov \cite{CG} (see also Lück's book \cite{Luck}). The case $p=1$ is slightly singular: $\ssl{\ell^1H}^1(G) = 0$ if and only if $G$ has zero or one end (see \cite[Appendix A]{moi-poiss}). The results presented here will cover other partial (positive) answers due to Kappos \cite{Kap}, Martin \& Valette \cite{MV} and Tessera \cite{Tes}.

The problem is reduced to a question of transport cost for measures tending to a left-invariant mean in $G$.

\begin{dfn}
A finitely generated group $G$ is called transport amenable if there exists $S$ a finite generating set, a positive constant $K \in \rr$ and a sequence of finitely supported probability measure $\xi_n$ converging to a left-invariant mean such that, $\forall s \in S$ and $\forall n$, the transport cost
\[
\trc(\rho_s \xi_n^\vee, \xi_n^\vee) = \inff{m \in M(\rho_s \xi_n^\vee, \xi_n^\vee)}  \int_{G \times G} d_\Gamma(x,y) \dd m(x,y) \leq K,
\]
where $d_\Gamma$ is the distance in the Cayley graph $\Gamma = \cay(G,S)$, $\xi^\vee(x) = \xi(x^{-1})$ (so that $\rho_s \xi^\vee = (\lambda_s \xi)^\vee$),  and $M(\xi,\phi)$ is the collection of finitely supported measures on $G \times G$ with marginals $\xi$ and $\phi$. \hfill $\clubsuit$
\end{dfn}
\begin{teo}\label{tvantran-t}
Let $G$ be a finitely generated group. If $G$ is transport amenable then the degree one reduced $\ell^p$-cohomology of $G$ vanishes for all $p \in ]1,\infty[$. 
\end{teo}
The proof relies almost solely on classical functional analysis; the exception being a lemma from Holopainen \& Soardi \cite[Lemma4.4]{HS}.

This theorem extends a result of Tessera \cite[Theorem 2.2]{Tes} (when the $L^p$-representation is the left-regular representation, see Theorem \ref{tconfol-t} in the present text) concerning groups with controlled F{\o}lner sequence (henceforth abbreviated by CF). Tessera's result implies these have trivial reduced degree one $\ell^p$-cohomology. The simplest example of a F{\o}lner sequence which is not CF is a sub-sequence of balls in a group of intermediate growth.

The exact range (in the class of amenable groups) of Tessera's result (and, consequently, of Theorem \ref{tvantran-t}) is unclear, though Tessera's article \cite{Tes-iso} gives a temptingly complete description of such groups. This class contains groups of polynomial growth, polycyclic groups and wreath products $F \wr \zz$ where $F$ is a finite group.

However, Theorem \ref{tvantran-t} applies to groups where it is known that \cite[Theorem 2.2]{Tes} does not apply. To see this, one must invoke the results of Erschler \cite{Ers} (or look again in \cite{Tes-iso}). Indeed, groups with CF have an isoperimetric profile in $\log v$, so that $\zz \wr \zz$ and $\zz_2 \wr \zz^2$ may not have CF. 

In examples \ref{ex-reverb} and \ref{ex-reverb2}, a small class of groups is shown to be transport amenable (though they are not CF, as explained in the previous paragraph). These groups are wreath products $H' \wr H$ where $H'$ is finitely generated amenable (finite or infinite) and $H$ is finitely generated infinite Abelian. In particular, this means that Theorem \ref{tvantran-t} applies to $\zz \wr \zz$ and $\zz_2 \wr \zz^2$.

A short discussion is done (as an appendix) to show that the full strength of Tessera's result (\ie the vanishing of the reduced cohomology for mixing $L^p$-representations) can be  obtained for $H \wr \zz^k$ (where $H$ is finitely generated amenable and $k>0$). It does not seem straightforward to decide whether it holds for arbitrary transport amenable groups.

It was known that groups with infinitely many finite conjugacy classes (\eg an infinite centre, \eg groups of polynomial growth\footnote{Groups 
of polynomial growth are virtually nilpotent by Gromov's famous result \cite{Gro-polyn}, hence quasi-isometric to a nilpotent group. But nilpotent groups possess infinitely many finite conjugacy classes.}) 
have trivial reduced $\ell^p$-cohomology in degree $1$ (see Kappos' preprint \cite[Theorem 6.4]{Kap}; the case of infinite centre is explained in Gromov \cite[\S{}8]{Gro} and detailed in Puls \cite[\S{}5]{Puls-prem} and Martin \& Valette's paper \cite[Theorem 4.3]{MV}). Note that the same method used for Theorem \ref{tvantran-t} may be applied to recover this vanishing result, see Theorem \ref{tcentrinf-t}).

In \cite{moi-poiss} the author developed another method to show (among other results) that groups with trivial Poisson boundary (for a SRW on the Cayley graph, \ie Liouville) and superpolynomial growth have trivial reduced $\ell^p$-cohomology (see \tpoiss). Though all groups with CF have trivial Poisson boundary, groups like $\zz \wr \zz$ and $\zz_2 \wr \zz^3$ have non-trivial Poisson boundary\footnote{Wreath 
products $H' \wr H$ are also never finitely presented, unless $H' = \{1\}$ or $H$ is finite, see \cite{Baumslag-wrpasfinpres}.} (see Kaimanovich \cite[Theorem 3.3]{Kaimano}).
Thus Theorem \ref{tvantran-t} is not covered by \tpoiss.

The methods presented here may be of use in non-amenable groups, but apart for groups with infinitely many finite conjugacy classes, the author could not find any other case. They also apply to graphs, but the proper conditions on graphs are not so convenient to formulate (\eg quasi-transitive action by quasi-isometries).

When $1<p<\infty$, it is known (see Puls \cite[\S{}3]{Puls-harm} or Martin \& Valette \cite[\S{}3]{MV}) that the existence of non-constant $p$-harmonic function (\ie $h \in \D^p(\Gamma)$ with $\nabla^* \maza_{p,p'} \nabla h = 0$, where $\maza_{p,p'}$ is the Mazur map defined by $(\maza_{p,p'} f)(\gamma) = |f(\gamma)|^{p-2} f(\gamma)$) is equivalent to the non-vanishing of reduced cohomology in degree $1$. Other known consequences of the triviality of the reduced $\ell^p$-cohomology include the triviality of the $p$-capacity between finite sets and $\infty$ (see \cite{Puls-pharmbnd} and \cite{Yam}).

\begin{acknowledgements}
The author is grateful to M.~de~la~Salle, P.~Pansu, J.~C.~Sikorav, R.~Tessera as well as the referee for comments and corrections.
\end{acknowledgements}

%------------------------------------------
\section{Notations and preliminary definitions}\label{s-def}
%------------------------------------------

Given a finitely generated group $G$ and a finite generating set $S$, the Cayley graph $\cay(G,S)$ is the graph whose vertices are the element of $G$ and $(\gamma,\gamma') \in E$ if $\exists s \in S$ such that $s^{-1} \gamma = \gamma'$. (Consequently, all Cayley graphs will be assumed to be of bounded valency.) In order for the resulting graph to have a symmetric edge set, $S$ is always going to be symmetric (\ie $s \in S \imp s^{-1} \in S$). Also, Cayley graphs are always going to be assumed connected (\ie $S$ is generating). 

For a generic graph $\Gamma = (X,E)$. The set of edges is thus thought of as a subset of $X \times X$. 
Functions will take value in $\kk$, a complete normed field where classical theorems about Banach spaces hold, \eg $\rr$ or $\mathbb{C}$. 
Thus $\ell^p(G)$ is the Banach space of functions on the vertices which are $p$-summable, while $\ell^p(E)$ will be the subspace of %(anti-symmetric) 
functions on the edges which are $p$-summable. 

The gradient $\nabla:\kk^X \to \kk^E$ is defined by $\nabla g(\gamma,\gamma') = g(\gamma') - g(\gamma)$, as the graphs are of bounded valency it is a bounded (and injective) operator from $\ell^p(G)$ into $\ell^p(E)$. Its kernel in $\kk^G$ is the space of constant functions. It is worthwhile to observe that the gradient is made of $ \{(\lambda_s - \Id) g\}_{s \in S}$ where $\lambda$ is the left-regular representation. As for the right-regular representation, it is a (injective) homomorphism from $G$ into $\mathrm{Aut}\big(\cay(G,S) \big)$, the automorphism group of the Cayley graph.

The Banach space of $p$-Dirichlet functions is the space of functions $f$ on $X$ such that $\nabla f \in \ell^p(E)$. It will be denoted $\D^p(\Gamma)$. In order to introduce the $\D^p(\Gamma)$-norm on $\kk^X$, it is necessary to choose a vertex $e$ (in a Cayley graph, it is convenient to choose the neutral element). This said $\|f\|_{\D^p(\Gamma)}^p = \|\nabla f\|_{\ell^p(E)}^p + |f(e)|^p$. By taking the primitive of the gradients, one may also prefer to define the reduced $\ell^p$-cohomology as 
\[
\ssl{\ell^p H}^1(\Gamma):= \D^p(\Gamma) / \srl{ \ell^p(X) + \kk}^{\D^p}. 
\]

Lastly, $p'$ will denote the H\"older conjugate exponent of $p$, \ie $p' = p/(p-1)$ (with the usual convention that $1$ and $\infty$ are conjugate). 

A sequence of finite sets $\{F_n\}$ in $G$ will be called a left-F{\o}lner sequence (resp. right-F{\o}lner sequence) if
\[
\forall \gamma \in G, \frac{|\gamma F_n \Delta F_n|}{|F_n|} \tosp{n} 0, \quad \textrm{ (resp. } \forall \gamma \in G, \frac{|F_n \gamma \Delta F_n|}{|F_n|} \tosp{n} 0 \textrm{ )}.
\]
A finitely generated group is amenable if and only if it has a F{\o}lner sequence. Following Tessera \cite{Tes}, the F{\o}lner sequence will be called \emph{controlled}, if the above ratio is less than $K (\diam F_n)^{-1}$ for some constant $K$. 

Existence of a left- (resp. right-) F{\o}lner sequence is equivalent to the existence of an invariant mean, \ie a positive, linear, continuous, left- (resp. right-) translation invariant map $\mu: \ell^\infty(G) \to \rr$, often normalised so that $\mu(\un_G)=1$. This, in turn, implies the existence of a sequence of finitely supported positive probability measures $\xi_n$ satisfying $\|\lambda_\gamma \xi_n - \xi_n\|_{\ell^1(G)} \to 0$ (resp. $\|\rho_\gamma \xi_n - \xi_n\|_{\ell^1(G)} \to 0$). It is a standard trick to construct a bi-invariant (\ie left- \emph{and} right-) version of any of the previous items from a left- (resp. right-) invariant one. There are many possible references on this topic, \eg Paterson's \cite{amen} or Pier's \cite{amen2} book.

Throughout this text, $\delta_\gamma$ is the Dirac mass at $\gamma$: $\delta_\gamma(\eta) = 0 $ if $\eta \neq \gamma$ and $1$ if $\eta = \gamma$. The following convention for the (left-)convolution of function will be in use:
\[
\xi * f (\eta) = \sum_{\gamma \in G} \xi(\gamma) \lambda_\gamma f(\eta) = \sum_{\gamma \in G} \xi(\gamma) f(\gamma^{-1} \eta) 
\]
Also, since this paper deals exclusively with Cayley graphs of a group $G$, the vertex set $X$ is equal to (and thus replaced by) $G$.

\setcounter{teo}{0}
\renewcommand{\theteo}{\thesection.\arabic{teo}}

\section{A criterion for convergence in $\D^p$}\label{ss-trique}
%-----------------------------------------------------------

Assume $\Gamma = \cay(G,S)$ is the Cayley graph of $G$ for the finite generating set $S$. The gradient $\nabla$ is made up of the partial gradients $\nabla_s:\kk^G \to \kk^G$ defined by 
\[
\nabla_s f = (\lambda_s - \Id) f = (\delta_s - \delta_e) * f. 
\]
Thus, if $f \in \D^p(\Gamma)$ then $\lambda_\gamma f$ represents the same class in $\ssl{\ell^p H}^1(\Gamma)$. Indeed, $\lambda_s f - f = \nabla_s f \in \ell^p(G)$ so, writing $\gamma$ as a word in $S$, $\lambda_\gamma f$ is also in the same class. Furthermore, a convex combination of such functions will always represent the same class. Actually, if $\xi \in \kk(\Gamma)$ is a function of finite support with $\sum_{\gamma \in G} \xi(\gamma) = 1$ then $[\xi * f] = [f]$.

Next recall that if a sequence $f_n \in \D^p$ converges point-wise (that is its value at $e$ and the values of the gradient on the edges) to $0$ and is bounded, then it weak$^*$-converges to $0$. This is relatively general fact. Indeed, assume $f_n \in \ell^p(\nn)$ is a bounded sequence which converges point-wise. Let $h$ be in the pre-dual, for any $\eps >0$, there is a finite set $F \subset \nn$ such that, if $h|_F$ is the restriction of $h$ to $F$, then $\|h - h|_F\| < \eps$. In particular, $|\pgen{f_n,h}| < |\pgen{f_n,h|_F}| + \|f_n\| \eps \leq |\pgen{f_n, h|_F}| + K \eps$.

Next, if $p \in ]1,\infty[$, weak$^*$-convergence of $f_n$ implies weak-convergence of $f_n$, which in turns implies norm-convergence of some sequence $h_k$ where each $h_k$ is a convex combination of the $f_n$ for $n \in \{1,2,\ldots ,k\}$ (a consequence of the Hahn-Banach theorem, see \cite[Theorem 3.13]{Rud}). Putting these two facts together one gets
\begin{lem}\label{tcondvan-l}
Let $f \in \D^p(\Gamma)$. If there exists a sequence $\xi_n$ of finitely supported functions such that:
\begin{enumerate}\renewcommand{\labelenumi}{{\normalfont \bf \arabic{enumi}.}}
\item $\sum_{\gamma \in G} \xi_n = 1$;
\item $\forall s \in S, (\delta_s - \delta_e) * \xi_n * f$ converge point-wise to $0$; 
\item and $\exists K>0$ such that $\forall s \in S, \| (\delta_s - \delta_e) * \xi_n * f\|_{\ell^p(G)} <K$,
\end{enumerate}
then $[f] = 0 \in \ssl{\ell^p H}^1(\Gamma)$.
\end{lem}
\begin{proof}
Indeed if {\bf 1} holds, then all the $\xi_n * f$ represent the same cohomology class. One may also add a constant function at every step, so that in addition to {\bf 2} these converges point-wise (at $e$). {\bf 3} implies that up to taking a convex combination, there is a sequence $y_n$ with $[y_n]=[f]$ (being convex combinations of $\xi_n * f$ and since $[\xi_n * f] = [f]$) and $y_n \to 0$ in the norm of $\D^p(\Gamma)$. This implies that the reduced class is trivial. 
\end{proof}
This gives the following theorem, which may also be found in Kappos' preprint \cite{Kap} (where the conclusion extends to higher degrees if $G$ is of type $\FP_n$). %or Martin \& Valette \cite[Theorem 4.2]{MV} (the hypothesis there is either infinite centre in $G$ or a non-amenable normal subgroup $N$ with infinite centraliser $Z_G(N)$).
\begin{teo}\label{tcentrinf-t}
Let $G$ be a finitely generated group with infinitely many finite conjugacy classes. Then $\ssl{\ell^p H}^1(\Gamma)=0$ for any Cayley graph $\Gamma$ of $G$. 
\end{teo}
\begin{proof}
Take a sequence of $C_n \subset G$ of finite conjugacy class such that the distance to the identity tends to infinity and let $\xi_n = \un_{C_n}/|C_n|$ be the normalised characteristic function. Then, the condition of the above lemma are satisfied for any $f \in \D^p(\Gamma)$ as 
\begin{enumerate}\renewcommand{\labelenumi}{{\normalfont \bf \arabic{enumi}.}}
\item $\xi_n$ is positive and has $\|\xi_n\|_{\ell^1(G)} =1$;
\item Functions in $\ell^p(E)$ decay at infinity. Since  $(\delta_s -\delta_e) * \xi_n * f = \xi_n * (\delta_s -\delta_e) * f$ the point-wise convergence to $0$ follows;
\item Again $(\delta_s -\delta_e)* \xi_n * f = \xi_n * (\delta_s -\delta_e)*  f$ and since $\|\xi_n\|_{\ell^1(G)} = 1$ it has norm $1$ as a convolution on $\ell^p$, thus Lemma \ref{tcondvan-l} may be applied with $K=1$.
\end{enumerate}
Since the above holds for any $f \in \D^p$, the $\ell^p$-cohomology ($1<p<\infty$) of groups with infinitely many finite conjugacy classes is trivial.
\end{proof}
As noted by Kappos in \cite{Kap}, the previous theorem covers virtually nilpotent groups. For more amusing examples of groups with this property, see \cite{BCD}.

By a lemma of Holopainen \& Soardi \cite[Lemma 4.4]{HS} (and the equivalence between non-existence of non-constant $p$-harmonic maps and vanishing of the reduced $\ell^p$-cohomology in degree one due to Puls \cite[\S{}3]{Puls-harm}, see also Martin \& Valette \cite[\S{}3]{MV}), it is actually sufficient to show the vanishing of cohomology on bounded functions, \ie for $f \in \BD^p(\Gamma) := \ell^\infty(G) \cap \D^p(\Gamma)$. This is quite useful, as one is tempted to take $\xi_n$ to be a sequence which tends to an invariant mean. 

\begin{lem}\label{tconvoper-l}
Let $G$ be an amenable group, $\Gamma= \cay(G,S)$ and $\xi_n$ is a sequence of measures tending to a left-invariant mean. If there is a $K>0$ such that 
\[
\forall s \in S, \quad \| (\delta_s - \delta_e) * \xi_n * \cdot \|_{\D^p(\Gamma) \to \ell^p(G)} \leq K
\]
then $\ssl{\ell^pH}^1(\Gamma)$=0.
\end{lem}
\begin{proof}
Indeed, assume $f \in \BD^p(\Gamma)$, then the first two points of Lemma \ref{tcondvan-l} are automatically satisfied. Indeed, \textbf{1} is a consequence that $\xi_n$ is a probability measure and \textbf{2} follows from
\[
|(\delta_s - \delta_e) * \xi_n * f (\eta)| \leq \|\lambda_s \xi_n - \xi_n \|_{\ell^1(G)} \|f\|_{\ell^\infty(G)},
\]
as $\|\lambda_s \xi_n - \xi_n \|_{\ell^1(G)} \to 0$. Finally, the hypothesis is a restatement of \textbf{3}. 
\end{proof}
In general, a computation gives:
\[ %\eqtag \label{eq-calcon}
\begin{array}{rl}
(\delta_s - \delta_e) * \xi_n * f (\eta) 
&= \displaystyle \sum_{\gamma \in s X_n} \xi_n(s^{-1} \gamma) \lambda_\gamma f(\eta) - \displaystyle \sum_{\gamma \in X_n} \xi_n(\gamma) \lambda_\gamma f(\eta)  \\
&= \displaystyle \sum_{\gamma \in X_n} \xi_n(\gamma) f(\gamma^{-1} s^{-1} \eta) - \displaystyle \sum_{\gamma \in X_n} \xi_n(\gamma) f(\gamma^{-1} \eta),
\end{array}
\]
where $X_n$ is the support of $\xi_n$. In order to express the above sum in terms of the gradient, it is necessary to express the difference as integrating the gradient on the paths. In the simpler case where $\xi_n = \un_{F_n} / |F_n|$ and $F_n$ is a left-F{\o}lner sequence, the above sum is the sum of values of $f$ on $F_n^{-1} s^{-1} \eta \setminus F_n^{-1} \eta$ minus the values of $f$ on $ F_n^{-1} \eta \setminus F_n^{-1} s^{-1} \eta$, all this divided by $|F_n|$. So, from now on, the aim will be to bound the transport of a uniform mass on $F_n^{-1} s^{-1} \eta \setminus F_n^{-1} \eta$ towards a uniform mass on $ F_n^{-1} \eta \setminus F_n^{-1} s^{-1} \eta$ if the cost of transport is given by the gradient of $f$.

\section{Proof of the Theorem \ref{tvantran-t}}\label{ss-preuve}
%-------------------------------------------------

The previous computation is an incentive to look at how one transports a probability measure to another.
\begin{dfn}
Let $\xi$ and $\phi$ be two finitely supported functions with $\sum_{\gamma \in G} \xi(\gamma) = \sum_{\gamma \in G} \phi(\gamma)$. Fix a ``price'' function $f':E \to \kk$ on the edges (not necessarily positive). Let $\wp$ be the set of (finite) paths in the graph. 
Let $M_\wp(\xi,\phi)$ be the collection of (finitely supported) measures on $\wp$ such that the marginals are $\xi$ and $\phi$, \ie
\[
\sum_{P \in \wp, \textrm{source}(P)=x} m(P) = \xi(x) \qquad \textrm{and} \qquad \sum_{P \in \wp, \textrm{target}(P)=y} m(P) = \phi(y).
\]
Let $m \in M_\wp(\xi,\phi)$, $m$ is a transportation pattern or, for short, a pattern. Define the ``flattened'' measure $m^\flat$ on edges by the quantity of mass transported through that edge, \ie
\[
a \in E, \quad m^\flat(a) = \sum_{P \ni a} m(P),
\]
where the sum is to be made with multiplicities (if $a$ appears twice in $P$). The transportation cost of the measure $\xi$ to the measure $\phi$ according to the pattern $m$ and price $f'$ is
\[
\trc_{f',m}(\xi,\phi) =  \int_{E}  f'(a) \dd m^\flat(a).
\]
The infimal transportation cost of the measure $\xi$ to the measure $\phi$ is 
\[
\begin{array}{rll}
\trc(\xi,\phi) 
& := \displaystyle \inf_{m \in M_\wp(\xi,\phi)} \trc_{\un,m}(\xi,\phi) 
& = \displaystyle \inf_{m \in M_\wp(\xi,\phi)} \displaystyle  \int_{E} 1 \dd m^\flat(a) \\
& = \displaystyle \inf_{m' \in M_G(\xi,\phi) } \displaystyle \int_{G \times G} d_\Gamma(x,y) \dd m'(x,y).
\end{array}
\]
where $\un$ is the function with constant value $1$, $M_G(\xi,\phi)$ is the space of finitely supported measures on $G \times G$ with (usual) marginals $\xi$ and $\phi$ and $d_\Gamma(x,y)$ is the graph distance. \hfill $\clubsuit$
\end{dfn}
Assume without loss of generality that we are dealing with probability measures. Notice that if $f' = \nabla f$, then for any pattern $m$,
\[
\begin{array}{rl}
\trc_{\nabla f,m}(\xi,\phi) 
  &= \displaystyle \int_{E} \nabla f(a) \dd m^\flat(a) \\
  &= \displaystyle \int_{\wp}  \sum_{a \in P} \nabla f(a) \dd m(P) \\
  &= \displaystyle \int_{\wp}  ( f(\text{target}(P)) - f(\text{source}(P)) ) \dd m(P) \\
  &= \displaystyle \int_{G} f(y) \dd \phi(y) - \int_{G} f(x) \dd \xi(x). \\
\end{array}
\]
Using this language one has that, for any pattern $m$, 
\[
(\delta_s - \delta_e) * \xi_n * f (\eta) = \trc_{\rho_\eta \nabla f,m}\big( (\lambda_s \xi_n)^\vee, \xi_n^\vee \big)
\]
where $\rho_\eta$ is the right-regular representation and $\xi^\vee(\gamma) = \xi(\gamma^{-1})$ (so $(\lambda_s \xi)^\vee = \rho_s \xi^\vee$). 

\begin{proof}[Proof of Theorem \ref{tvantran-t}]
The goal is to obtain the bound in Lemma \ref{tconvoper-l}. By hypothesis, for every $n$, there is a pattern $m_n$ such that $\trc_{\un,m_n}(\rho_s \xi_n^\vee, \xi_n^\vee) \leq 2K$. For each $n$, this pattern will be used to get a bound on the map $T$ which sends a function $f' \in \ell^p(E)$ to $(Tf')(\eta) = \trc_{\rho_\eta f',m_n}\big( (\lambda_s \xi_n)^\vee, \xi_n^\vee \big)$ (which will be in $\ell^p(G)$). The bound will be obtained for $p=1$ and $p=\infty$, and deduced by Riesz-Thorin interpolation for intermediate $p$. In an attempt to alleviate notations the index $n$ will be omitted.

To get a $\ell^\infty$ bound, it suffices, to look at $\trc_{\un,m}(\rho_s \xi^\vee, \xi^\vee)$: right multiplication is an automorphism and replacing $f'$ by $\| f'\|_{\ell^\infty(E)}$ on each edge may only increase these sums. The $\ell^1$-bound is in appearance more intricate as one should not dispose so carelessly of the values of $f'$. Indeed,  if the weight function on the edges is $f' \in \ell^1(E)$ then the graph can be seen as a bounded (but not compact) metric space (\emph{via} $|f'|$).

The question can then be reduced to determining whether it is possible to solve 
%(with $\xi = \rho_s \xi_n^\vee$ and $\phi = \xi_n^\vee$)
the transport problem (and all its translates) so that for any $f' \in \ell^1(E)$ and $\forall a \in E, \sum_{\eta \in G} m^\flat_{\rho_\eta f'}(a) < K \| f' \|_{\ell^1(E)}$ where $m^\flat_{g'}$ is defined by $m^\flat_{g'}(a) = g'(a) m^\flat$, \ie $g' \dd m^\flat$. This seems quite daunting. However, this bound is automatic if one has that $m^\flat_{\un} (E) \leq K$ (recall $m^\flat_{\un} (E) = m^\flat(E)$). Indeed, if one knows that the total quantity of mass is transported over all edges is bounded in one transport, then in all translations, any given edge will be used at most $|S|$ times this amount (the factor $|S|$ is necessary as the action of $G$ on non-oriented edges is not always faithful). In other words,
\[
\sum_{\eta \in G} m^\flat_{\rho_\eta f'}(a) \leq |S| \trc_{\un,m} (\xi,\phi)  \| f' \|_{\ell^1(E)}. 
\]
Thus, using the Riesz-Thorin theorem, one gets, for all $p \in [1,\infty]$, a uniform bound on $\| (\delta_s - \delta_e) * \xi_n * \cdot \|_{\D^p(\Gamma) \to \ell^p(G)}$. Applying Lemma \ref{tconvoper-l}, the only class in reduced cohomology is the trivial one. 
\end{proof}

Note that using $\xi_n = \un_{F_n} / |F_n|$ in Theorem \ref{tvantran-t} (here $F_n$ are F{\o}lner sets), one sees immediately vanishing of cohomology for Abelian groups: there is a bijection between $F_n^{-1} s = s F_n^{-1}$ and $F_n^{-1}$ which sends every vertex to a neighbour using only one edge, so the above may be applied with $K=1$.

The hypothesis of Theorem \ref{tvantran-t} are always satisfied when $G$ has a controlled F{\o}lner sequence, \ie a sequence of finite sets such that $\forall s \in S,$
\[
\frac{|F_n \Delta s F_n|}{|F_n|} \leq \frac{K}{ \diam F_n }.
\]
Since $F_n$ may be translated to be contained in the ball of radius $\diam F_n$. This gives back a special case (when the $L^p$-representation is the left-regular representation) of a result of Tessera \cite[Theorem 2.2]{Tes}:
\begin{teo}\label{tconfol-t}
If $G$ is a finitely generated group with a controlled F{\o}lner sequence, then its degree one reduced $\ell^p$-cohomology is trivial.
\end{teo}
\begin{proof}
If $\xi_n = \un_{F_n} /|F_n|$ one sees, up to translating $F_n$ into a ball of radius $\diam F_n$ around $e_G$, that the transport problem may be solved by picking any bijection from $F_n^{-1} s \setminus F_n^{-1}$ to $F_n^{-1} \setminus F_n^{-1} s$ and looking at the corresponding transport cost (at most $(2 \diam F_n + 2) |s F_n \Delta F_n| / |F_n|$).
\end{proof}

%----------------
\section{Examples}\label{ss-exemple}
%----------------

However, there are examples of groups which are not covered by Theorem \ref{tconfol-t} but where Theorem \ref{tvantran-t} applies. This is the case for many wreath products. The conventions for the wreath product (of two finitely generated groups) $H' \wr H$ will be as follows. $H'$ is the ``lamp state'' group and $H$ is the ``lamplighter position'' group. Their wreath product will be the semi-direct product $ H \ltimes_\rho (\oplus_H H')$ where $H$ acts on the infinite sum $\oplus_H H'$ by shifting indices: $\rho_h (\phi)(k) = \phi(kh) $

The generators are going to be of the form:  $S= \{e_H \} \times \{e_H \to S_{H'}\}  \cup S_H \times \{0\}$ where $S_{H'}$ and $S_H$ are some set of generators for $H'$ and $H$ (respectively), $\{e_H \to S_{H'} \}$ is the set of functions supported on $e_H$ with value in $S_{H'}$ and $\{0\}$ is the trivial function (\ie sending all $H$ to $e_{H'}$). 

Multiplication on the left by $(e_H \to s_{H'}, e_H )$ changes the lamp state where the lamplighter is, whereas multiplication on the right changes the lamp state at the position $e_H$. Multiplication on the left by $(0, s_H)$ moves the lamplighter, whereas multiplication on the right moves the lamplighter together with all the lamps. To put this all in one formula: let $h,g \in H$ and $\phi,\psi: H \to H'$ be finitely supported, then
\[
(h, \{\phi_k\}_{k \in H}) \cdot (g, \{\psi_k\}_{k \in H}) = ( h g , \{ \phi_{kg} \psi_k \}_{k \in H})
\]

\begin{exa}\label{ex-reverb}
Let $\Gamma$ be the Cayley graph of $H' \wr H$ where $H'$ is finite and $H$ is Abelian. 

A F{\o}lner sequence is given as follows: let $A_n$ be a controlled F{\o}lner sequence of diameter $d_n$ in $H$ and let $B_k$ be the functions $H \to H'$ which are supported on $A_k$ (\ie all the non-trivial lamp states supported by $A_k$). Then $F_n := B_n \times A_n$ is a (left- but not right-) F{\o}lner sequence. The transport problem is particularly easy to solve for $\xi_n = \un_{F_n} / |F_n|$, the normalised characteristic function of the F{\o}lner sets. 

Indeed, $r \in \{e_H \to S_{H'} \} \times \{e_H\}$ does not displace the F{\o}lner set at all. If $z \in \{0\} \times S_H$, then the problem is to displace elements from $F_n^{-1} z$ to $F_n^{-1}$ with a cost which is bounded by a constant multiple of $|F_n|$. 

Since the $F_n$ are particularly easy to describe and the $F_n^{-1}$ more ungainly, it turns out to be easier to write down the problem by inverting everything: namely displace $z^{-1} F_n$ to $F_n$, where the cost of displacing $\gamma$ to $\gamma'$ is the length of the element $\eta$ such that $\gamma \eta = \gamma'$.

Elements in  $z^{-1} F_n \setminus F_n$ are those where the lamplighter moved out of $A_n$. Since $H$ is assumed Abelian, for elements where the lamp states are all off on $A_n \setminus z^{-1} A_n$ it is possible to  displace all these elements back to elements in $F_n$ by multiplying on the right by $z$. This will contribute to the final total cost (after division by $|F_n|$) less than $\dfrac{|A_n \setminus z^{-1} A_n| }{ |A_n| c^{|A_n \setminus z^{-1}A_n|}}$.

The aim is now to send elements with some non-trivial lamp states on $\Delta_z A_n := z A_n \setminus A_n$ to similar elements with non-trivial lamp states on $\Delta'_z A_n := A_n \setminus z A_n $. Let $\phi: A_n z \setminus A_n \to A_n \setminus A_n z$ be a bijection. Proceed as follows:
\begin{itemize}
\item If there are $i$ non-trivial lamp states, take the lamplighter (and the lamps, since the multiplication is on the right) along a path so that the $i$ lamps (which are on a non-trivial state) eventually pass at the identity. When these lamps are at the identity, interrupt this path so as to change their states to the trivial one (recall that multiplication on the right by a lamp-state element changes the state of the lamp at identity, not where the lamplighter is). Repeat a similar path (interrupted by lamp-state changes) to turn the $i$ corresponding (by $\phi$) lamps into the non-trivial states. When this is done, return the lamplighter to a neighbouring position; this takes at most 
%\footnote{A more efficient way to do this is to look at one path which go through all the concerned lamps, starts at the vertex of interest and ends at a neighbouring vertex. As in the paper from Naor-Peres \cite{NP}, this is mostly interesting if one deals with iterated wreath products.}
$2i(d_n+k)+d_n$ steps as $A_n$ is of diameter $d_n$ and the diameter of the ``lamp states group'' $F$ is $k$; 
\item For $i$ fixed, the number of such elements is $c^{|A_n|-|\Delta_z A_n|} (c-1)^i \binom{|\Delta_z A_n|}{i}$ where $c=|F|$;
\item So the total of used edges is at most (recall $\sum_{i=1}^{N} (c-1)^i i \binom{N}{i} = N c^{N-1}$)
\[
\begin{array}{l}
c^{|A_n|-|\Delta_z A_n|} \displaystyle \sum_{i=0}^{|\Delta_z A_n|} (c-1)^i \binom{|\Delta_z A_n|}{i} \big( 2i(d_n+k) + d_n ) \\
\qquad \leq c^{|A_n|-1} \bigg( 2 (d_n+k) |\Delta_z A_n| + d_n c \bigg);
\end{array}
\]
\item As $|F_n| = c^{|A_n|} |A_n|$, one gets a bound of $\dfrac{2 |\Delta_z A_n| (d_n + k) + d_n c}{|A_n| c}$.
\end{itemize}
Since $A_n$ may be chosen to be a controlled F{\o}lner sequence ($H$ is assumed Abelian), $d_n / |A_n| \leq 1$ and $k/c \leq 1$ is fixed, this is uniformly bounded by $1+\eps+2K/c$ where $\eps > 0$ is arbitrarily small and $K$ comes from the controlled F{\o}lner ratio. (Actually, if $H \neq \zz$, then the constant is $\eps+2K/c$.) Hence, Theorem \ref{tvantran-t} applies.

The sequence $F_n$ is not itself a controlled F{\o}lner sequence (unless $H = \zz$): $\diam F_n \geq |A_n|$, $|F_n| = |A_n| k^{|A_n|}$ and $|z F_n \Delta F_n| \geq |\Delta_z A_n| k^{|A_n|}$. \hfill $\Diamond$
\end{exa}

The strategy may very well extend to more general wreath products. Considering $\xi_n = \un_{F_n} *\un_{F_n^{-1}} / |F_n|^2$ to get some sort of bi-CF sequence might also be a good idea. The parameters $k=\diam H'$ and $c=|H'|$ have been carefully brought along in the above example, the reason being:

\begin{exa}\label{ex-reverb2}
Let $H'$ and $H$ be infinite amenable groups with respective F{\o}lner sequences $A_n'$ and $A_n$ and respective generating sets $S_{H'}$ and $S_H$. Let $d_n=\diam A_n$ and $k_i=\diam A'_i$. Assume $H$ is Abelian and that there is a sequence $i_n \in \nn$ such that 
\begin{itemize}
\item $\forall z \in S_{H'},  \limm{n \to \infty} \dfrac{2(d_n+k_{i_n}) |\Delta_z A'_{i_n}|}{|A_n| \, |A'_{i_n}|} \leq 2$.
\item $\forall z \in S_{H},  \limm{n \to \infty} \dfrac{2(d_n+k_{i_n}) |\Delta_z A_n| + d_n |A'_{i_n}|}{|A_n| \, |A'_{i_n}|} \leq 2$.
\end{itemize}
Then Theorem \ref{tvantran-t} applies to $H' \wr H$ (for the generating set $S= \{e_H \to S_{H'}\} \times \{e_H \} \cup \{0\} \times S_H$). 

To see this repeat the preceding example with $F_n$ being functions with support in $A_n$ and taking values in $A'_{i_n}$. (In doing so, $k$ becomes $k_i=\diam A'_i$ and $c$ becomes $c_i=|A'_i|$). The first condition comes up when solving the transport problem for the generators $(\{e_H \to h'\},e_H)$ (which no longer leave the set invariant). There are $|\Delta_z A'_{i_n}|$ which escape the F{\o}lner set, and each of them may be returned by a cost of at most $2(d_n+k_{i_n})$. The second is used when solving the transport problem for the generators $(e_{H'},h)$. 

The required sequence $i_n$ always exists. The first ratio always tends to $0$ if $i_n$ grows slowly enough, and, if $A_n$ is taken to be CF, the second ratio tends to either $1$ (when $H =\zz$) or $0$ (otherwise). \hfill $\Diamond$
\end{exa}

This means one may apply Theorem \ref{tvantran-t} to $H' \wr \zz^k$ for any $H'$ amenable and $k >0$.

%%%%%%%%%%%%%%%%%%%%%%%%%%%%%%%%%%%%%%%%%%%%%%%%%%%%%%%%%%%%%%%%%%%%%%%%%
%  Erratum and addendum: own title page, page numbering and bibliography
%%%%%%%%%%%%%%%%%%%%%%%%%%%%%%%%%%%%%%%%%%%%%%%%%%%%%%%%%%%%%%%%%%%%%%%%%

\clearpage
\pagenumbering{arabic}
\setcounter{section}{0}
\setcounter{teo}{0}
\setcounter{teoA}{0}
\setcounter{equation}{0}
\setcounter{footnote}{0}
\renewcommand{\un}[0]{\mathbbold{1}}
\renewcommand{\dd}[0]{\ensuremath{\mathrm{d}}}
\pdfbookmark[0]{Erratum and addendum}{erratum}

\begin{center}
\Large Erratum and addendum to ``Vanishing of $\ell^p$-cohomology\\ and transportation cost''
\vspace*{1cm}

{\large Antoine Gournay}\footnote{The author gratefully acknowledges partial support from the ERC AdG grant 101097307.}
\end{center}

\vspace*{1cm}

\centerline{\textsc{Abstract}}

\begin{center}
\parbox{12cm}{{ \small
Examples 5.1 and 5.2 of [Bull. London Math. Soc. \textbf{46} (2014), 481--490] are incorrect: a wreath product $H_0 \wr H$ with $H_0 \neq \{1\}$ and $H$ infinite is transport amenable only if $H_0$ is finite and $H$ is virtually cyclic. In particular, $\zz \wr \zz$ and $(\zz/2\zz) \wr \zz^2$ are not transport amenable. The main theorem of that paper is not affected, but the only groups currently known to be transport amenable are those with a controlled F\o lner sequence.
We also show that, on a transport amenable group, every $1$-cocycle of sublinear growth with values in a uniformly bounded Banach module is an almost coboundary, which extends a result of Cornulier \& Tessera; together with a theorem of Thom, this gives the vanishing of the first reduced cohomology for uniformly bounded, strongly mixing representations on superreflexive Banach spaces.
\par\medskip\noindent\textbf{MSC2020:} 20J06 (primary), 20E22, 22D12, 43A07, 43A15 (secondary).
}}
\end{center}

%---------------------
\section{Introduction}\label{e:s-intro}
%---------------------

In \cite{e:moi-trans} the author introduced transport amenable groups (see Definition \ref{e:dtranspam}) and proved that their reduced $\ell^p$-cohomology in degree one vanishes for $1<p<\infty$ \cite[Theorem 1.3]{e:moi-trans}. Groups with a controlled F\o lner sequence (CF groups for short), \ie with finite sets $F_n$ contained in the balls of radius $r_n \to \infty$ around the identity and such that $|sF_n \,\Delta\, F_n| \leq C|F_n|/r_n$ for every generator $s$ (see Tessera \cite{e:Tes} and Cornulier, Tessera \& Valette \cite{e:CTV}), are transport amenable \cite[\S{}4]{e:moi-trans}; they include the polycyclic groups and the lamplighters $F \wr \zz$ with $F$ finite (see \cite{e:Tes} and \cite[\S{}6]{e:CT}). Examples 5.1 and 5.2 of \cite{e:moi-trans} claimed that the wreath products $H_0 \wr H$ with $H_0$ amenable and $H$ infinite abelian are transport amenable as well. This would have provided transport amenable groups which are not CF, such as $\zz \wr \zz$ and $(\zz/2\zz) \wr \zz^2$. These examples are incorrect.

\begin{teoA}\label{e:tA}
Let $H_0 \neq \{1\}$ and $H$ be finitely generated groups, with $H$ infinite. If $H_0 \wr H$ is transport amenable, then $H_0$ is finite and $H$ is virtually cyclic.
\end{teoA}

The proof is given in \S{}\ref{e:swreath}, together with an explanation of where the argument of \cite{e:moi-trans} fails (Remark \ref{e:rerror}). The consequences are the following.
\begin{itemize}
\item[$\bullet$] Example 5.2 of \cite{e:moi-trans} is false, and Example 5.1 holds only when $H$ is virtually cyclic; for $H = \zz$ the F\o lner sets used there form a controlled F\o lner sequence, as observed at the end of that example.
\item[$\bullet$] The assertions in the introduction of \cite{e:moi-trans} that its Theorem 1.3 applies to $\zz \wr \zz$, to $(\zz/2\zz) \wr \zz^2$ and to $H_0 \wr \zz^k$ for every amenable $H_0$, and that the vanishing of the reduced cohomology for mixing $L^p$-representations can be obtained for $H \wr \zz^k$, rely on these examples and are withdrawn.
\item[$\bullet$] Theorem 1.3 of \cite{e:moi-trans} and its proof are not affected. However, the only groups currently known to be transport amenable are the CF groups, for which the conclusion was already known by Tessera \cite{e:Tes}. Whether some transport amenable group is not CF is open (Question \ref{e:qcf}).
\item[$\bullet$] For the wreath products above, the vanishing of the reduced $\ell^p$-cohomology in degree one holds nevertheless, for other reasons: see \cite{e:Go-cras} and \cite{e:Go-lpharm}, as well as \cite[Theorem 1.1]{e:Go} for Liouville groups of superpolynomial growth.
\end{itemize}

The second purpose of this note is to record that transport amenability suffices for a result of Cornulier \& Tessera \cite[Proposition 5.6]{e:CT}, which states that, on a CF group, a $1$-cocycle with values in a uniformly bounded Banach module is an almost coboundary as soon as it grows sublinearly.

\begin{teoA}\label{e:tB}
Let $G$ be a finitely generated transport amenable group and let $(V,\pi)$ be a uniformly bounded Banach $G$-module.
\begin{enumerate}
\item[(a)] A $1$-cocycle $b : G \to V$ with sublinear growth is an almost coboundary.
\item[(b)] If $V$ is superreflexive and $\pi$ is strongly mixing (all its coefficients lie in $c_0G$), then $H^1_{\mr{red}}(G,\pi) = 0$.
\end{enumerate}
\end{teoA}

Part (a) is Theorem \ref{e:tsublin}; its converse holds on any group \cite[Proposition 5.6]{e:CT}. Part (b) follows by combining it with Thom's sublinearity theorem \cite[Theorem 1]{e:Thom} (Corollary \ref{e:cthomext}). It contains \cite[Theorem 1.3]{e:moi-trans}, and answers the question raised there of whether the vanishing holds for all mixing $L^p$-representations of transport amenable groups. In view of Theorem \ref{e:tA}, Theorem \ref{e:tB} goes beyond \cite{e:CT} and \cite{e:Thom} only if some transport amenable group is not CF.
The point of the proof is that a bound on the cost of a transport does not suffice: the mass which is moved must also tend to $0$, so that sublinearity can act on the long distances. This is automatic once transports are described as flows on the edges of the Cayley graph (\emph{transport patterns}), because an optimal pattern never moves mass which does not have to move (Lemma \ref{e:ldisplaced}).

\S{}\ref{e:stc} recalls transport costs and patterns, \S{}\ref{e:swreath} proves Theorem \ref{e:tA} and discusses obstructions to transport amenability, and \S{}\ref{e:ssublin} proves Theorem \ref{e:tB}.

{\itshape Acknowledgements:}
Claude Opus 5.5 was used to assist in drafting this note.
The content was reviewed and substantially revised by the author, who is responsible for the final text.

%---------------------------
\section{Transport cost and transport patterns}\label{e:stc}
%---------------------------

Throughout, $G$ is a finitely generated group, $S=S^{-1}$ a finite generating set, $|\cdot|$ the word length, and $\Gm = \cay(G,S)$ the Cayley graph whose edges are the pairs $(\gm,\gm s)$ with $s \in S$; its distance is $d_\Gm(x,y) = |x^{-1}y|$. Left translations, $\lambda_\gm\xi(x) = \xi(\gm^{-1}x)$, act on $\Gm$ by graph automorphisms; we also write $\rho_\gm\xi(x) = \xi(x\gm)$.

\begin{dfn}\label{e:dtc}
Let $\xi, \varphi: G \to \rr_{\geq 0}$ be finitely supported with
$\sum_\gm \xi(\gm) = \sum_\gm \varphi(\gm)$.
Let $\Mm(\xi,\varphi)$ be the set of finitely supported measures $m$ on $G \times G$ whose marginals are $\xi$ and $\varphi$, \ie
$\sum_y m(x,y) = \xi(x)$ and $\sum_x m(x,y) = \varphi(y)$.
The \textbf{transport cost} from $\xi$ to $\varphi$ is
\[
\TC(\xi, \varphi) \; := \inff{m \in \Mm(\xi,\varphi)} \sum_{(x,y) \in G\times G}
d_\Gm(x,y) \, m(x,y).
\]
\end{dfn}
This is the Kantorovich--Rubinstein (or Wasserstein-$1$, or ``earth mover'') distance for the metric $d_\Gm$.
Here are some properties: $\TC$ is a metric on finitely supported probability measures; it is
convex in each variable; it is invariant under any isometry of $\Gm$;
if $\xi \wedge \varphi (x) := \min\big( \xi(x), \varphi(x) \big)$, then $\TC(\xi,\varphi) = \TC(\xi - \xi \wedge \varphi,\varphi -\xi \wedge \varphi)$; lastly,
$\TC(\xi,\varphi) = \min \{ \|\tau\|_{\ell^1E} \mid \tau$ is finitely supported and $\nabla^* \tau = \varphi-\xi \}$.

The last description deserves a few words, since it is the one which will be used throughout.
The set of edges $E \subseteq G \times G$ is symmetric, and functions on $E$ are anti-symmetric.
The \textbf{gradient}
$\nabla : \rr^G \to \rr^E$ is $\nabla g(x,y) = g(y)-g(x)$,
and the \textbf{divergence} $\nabla^*: \rr^E \to \rr^G$ is its adjoint,
\ie it is defined by $\pgen{ \tau \mid \nabla g}_E = \pgen{ \nabla^* \tau \mid g}_G$,
which gives
\[
\nabla^* \tau (x) = \somme{y \in N(x)}{} \tau(y,x) \; - \somme{y \in N(x)}{}
\tau(x,y) .
\]

\begin{dfn}\label{e:dtp}
A \textbf{transport pattern} from $\xi$ to $\varphi$ is a finitely supported anti-symmetric $\tau : E \to \rr$ with $\nabla^*\tau = \varphi - \xi$.
Its \textbf{cost} is $\|\tau\|_{\ell^1E}$.
\end{dfn}

A transport pattern is just a way of making explicit along which paths of the graph the mass travels.
Note that both orientations of a geometric edge occur in $E$, so that the mass crossing an edge is $2\tau(x,y)$ and the cost counts each unit of mass once per edge it crosses.
For instance on $\zz$, with $\xi = \delta_0$ and $\varphi = \delta_1$, the only pattern is $\tau(0,1) = -\tau(1,0) = \tfrac12$, of cost $1$.

\begin{lem}\label{e:ltcpattern}
$\TC(\xi,\varphi) = \min \{ \|\tau\|_{\ell^1E} \mid \tau$ a transport pattern from $\xi$ to $\varphi\}$, and the minimum is attained by a pattern containing no cycle.
\end{lem}

\begin{proof}
Both sides are finite linear programmes and both have $\supp{f} \pgen{ f \mid \varphi - \xi}$, over $1$-Lipschitz $f$, as their dual:
for the left-hand side this is Kantorovich--Rubinstein duality,
for the right-hand side it is the duality between $\|\cdot\|_{\ell^1E}$ and $\|\nabla f\|_{\ell^\infty E} \leq 1$.
Concretely, a coupling $m$ yields a pattern by routing the mass $m(x,y)$ along a geodesic from $x$ to $y$, at a cost of at most $\sum d_\Gm(x,y)m(x,y)$;
conversely a pattern decomposes, by the flow decomposition, into finitely many weighted paths and cycles, and deleting the cycles leaves the divergence unchanged without increasing the cost.
\end{proof}

The reason for insisting on patterns rather than couplings is that the constraint $\nabla^*\tau = \varphi - \xi$ sees only the \emph{difference} of the two measures.
This is what makes the following automatic, and it is the point on which everything in \S{}\ref{e:ssublin} turns.

\begin{lem}\label{e:ldisplaced}
Let $\tau$ be a cycle-free transport pattern from $\xi$ to $\varphi$ and let $\gm_i$, carrying weights $c_i>0$, be a flow decomposition of it.
Then
\[
\somme{i}{} c_i = \tfrac{1}{2}\|\xi - \varphi\|_{\ell^1G}
\qquad \text{and} \qquad
\|\tau\|_{\ell^1E} = \somme{i}{} c_i \, |\gm_i| ,
\]
$|\gm_i|$ being the length of $\gm_i$. In particular, for every $R \geq 0$,
$
\somme{i \,:\, |\gm_i| \leq R}{} c_i \;\; \leq \;\; \tfrac{1}{2}\|\xi - \varphi\|_{\ell^1G} .
$
\end{lem}

\begin{proof}
A cycle-free decomposition has all its paths running from $\mr{supp}(\varphi-\xi)_-$ to $\mr{supp}(\varphi-\xi)_+$, so their total weight is the total excess $\tfrac12\|\xi-\varphi\|_{\ell^1}$.
The second identity is the definition of the cost, there being no cancellation between paths once the
cycles are gone.
The last assertion follows.
\end{proof}

Let us recall the origin of transport amenability \cite{e:moi-trans}. A function $f$ whose gradient is in $\ell^p$ belongs to the same class in reduced $\ell^p$-cohomology as its translates, and hence, by convex combinations, as its averages against any finitely supported probability measure. Averaging against measures which tend to an invariant mean, one would like to show that $f$ is in the same class as a constant function. There is however a question of norm convergence to handle, and this is where the transport cost comes in \cite[\S{}3--4]{e:moi-trans}.

\begin{dfn}\label{e:dtranspam}
$G$ is \textbf{transport amenable} if there are a finite generating set $S$, a constant $K$ and a sequence $\eta_n$ of finitely supported probability measures on $G$ such that, for all $s \in S$,
\[
\big\|\lambda_s\eta_n - \eta_n\big\|_{\ell^1G} \xrightarrow[\,n\to\infty\,]{} 0
\qquad\text{and}\qquad
\TC\big(\lambda_s\eta_n, \eta_n\big) \leq K \ \text{ for all } n.
\]
\end{dfn}

\begin{rmk}\label{e:rconv}
In \cite{e:moi-trans} the Cayley graph is defined by left multiplication, so that its distance is $d'(x,y) = |yx^{-1}| = d_\Gm(x^{-1},y^{-1})$, and \cite[Definition 1.2]{e:moi-trans} asks for measures $\xi_n$ tending to a left-invariant mean such that $\TC'(\rho_s\xi_n^\vee,\xi_n^\vee) \leq K$, where $\xi^\vee(x) = \xi(x^{-1})$ and $\TC'$ is computed with $d'$. Since $x \mapsto x^{-1}$ is an isometry from $(G,d')$ onto $(G,d_\Gm)$ and $(\rho_s\xi^\vee)^\vee = \lambda_s\xi$, this is Definition \ref{e:dtranspam} with $\eta_n = \xi_n$.

In Definition \ref{e:dtranspam}, the measures are moved by automorphisms of $\Gm$ (left translations) and have to be transported back along the edges of $\Gm$ (right multiplications). The two sides cannot be the same: $\rho_s\eta$ is obtained from $\eta$ by moving every point along an edge, so that $\TC(\rho_s\eta,\eta) \leq 1$ for every $\eta$. In the proof of Theorem \ref{e:tsublin}, the left translation comes from the cocycle identity $b(s\gm) = b(s) + \pi(s)b(\gm)$, and the distance $d_\Gm$ from the identity $b(y) - b(x) = \pi(x)b(x^{-1}y)$.
\end{rmk}

The property does not depend on the choice of $S$. This is a standard argument, writing $s \in S$ as a word in another generating set $S'$ and using the fact that $\TC$ is a metric which is invariant under the isometries of the graph.

Call the \textbf{``long displacement'' property} of a sequence of patterns the requirement that, for every $R$, the mass they carry along paths of length at most $R$ tends to $0$.
Lemma \ref{e:ldisplaced} says that this is \emph{not} an extra condition to be imposed in Definition \ref{e:dtranspam}: if $\tau_n$ is a cycle-free pattern of minimal cost from $\eta_n$ to $\lambda_s\eta_n$, then the total mass it carries is $\tfrac{1}{2}\| \lambda_s \eta_n - \eta_n \|_{\ell^1G} \to 0$, and the mass carried by the short paths is at most that.

What makes this worth stating is that the same is false for couplings, and that the discrepancy is only apparent.
Take $\zz$, with $\xi_n$ uniform on $\{-n,\ldots,n\}$ and $\varphi_n = \lambda_1\xi_n$.
One coupling $x \mapsto x+1$ displaces the whole mass over a distance $1$.
Another coupling fixes $\{-n+1,\ldots,n\}$ and displaces a mass $\tfrac{1}{2n+1}$ over a distance $2n+1$.
Both have cost $1$, and they induce the \emph{same} pattern, namely the one sending a mass $\tfrac{1}{2n+1}$ across each of the $2n+1$ edges between $-n$ and $n+1$.
Indeed, on $\zz$ there are no cycles, so the pattern with a given divergence is unique.
What Lemma \ref{e:ldisplaced} does is to select, among the readings of that one pattern, the one which is correct in general
(a single path of length $2n+1$ carrying a mass $\tfrac{1}{2n+1}$)
and this is precisely the behaviour of a controlled F\o lner sequence.
The first reading is an artefact of the coupling, and it is the reason why the proof of Theorem \ref{e:tsublin} is written in terms of patterns.

%---------------------------
\section{Wreath products}\label{e:swreath}
%---------------------------

For finitely generated groups $H_0$ (the lamp group) and $H$ (the base), the wreath product is $H_0 \wr H = \big(\oplus_H H_0\big) \rtimes H$. Its elements are the pairs $(\psi,k)$ of a finitely supported configuration $\psi : H \to H_0$ and a position $k \in H$ of the lamplighter, and
\[
(\phi,h)\,(\psi,k) = \big(\phi \cdot (h \cdot \psi),\, hk\big), \qquad \text{where } (h\cdot\psi)(y) = \psi(h^{-1}y).
\]
Let $S_{H_0}$ and $S_H$ be finite symmetric generating sets of $H_0$ and $H$, and write $|\cdot|_{H_0}$ for the word length of $H_0$. We use the generating set of $H_0 \wr H$ made of the \emph{lamp generators} $(\delta^a, e_H)$, $a \in S_{H_0}$, where $\delta^a$ is the configuration with value $a$ at $e_H$ and trivial elsewhere, and of the \emph{base generators} $(\mf{o},b)$, $b \in S_H$, where $\mf{o}$ is the trivial configuration. The edges of $\Gm$ are given by right multiplication:
\begin{itemize}
\item[$\bullet$] right multiplication by $(\delta^a,e_H)$ changes the lamp under the lamplighter, $\psi(k) \mapsto \psi(k)a$;
\item[$\bullet$] right multiplication by $(\mf{o},b)$ moves the lamplighter alone, $k \mapsto kb$.
\end{itemize}
The automorphisms of $\Gm$ used in Definition \ref{e:dtranspam} are left multiplications:
\begin{itemize}
\item[$\bullet$] left multiplication by $(\delta^a,e_H)$ changes the lamp at $e_H$, $\psi(e_H) \mapsto a\psi(e_H)$;
\item[$\bullet$] left multiplication by $(\mf{o},b)$ translates the whole picture, $(\psi,k) \mapsto (b\cdot\psi, bk)$.
\end{itemize}

The proof of Theorem \ref{e:tA} uses the following elementary fact; for random walks it is \cite[Proposition 3.5]{e:Go-TV}, and the proof given there applies to any measure.

\begin{lem}\label{e:lspeed}
Let $\Lambda$ be an infinite group with finite symmetric generating set $T$ and word length $|\cdot|$, and let $\nu$ be a finitely supported probability measure on $\Lambda$, with mean length $L := \sum_{g \in \Lambda} \nu(g)\,|g|$. Then
\[
\maxx{t \in T}\big\|\lambda_t\nu - \nu\big\|_{\ell^1\Lambda} \; \geq \; \frac{1}{22(1+L)} .
\]
\end{lem}
\begin{proof}
Put $\eps := \max_{t\in T}\|\lambda_t\nu-\nu\|_{\ell^1}$. Telescoping along a geodesic word, and using that $\|\cdot\|_{\ell^1}$ is invariant under left translations, gives $\|\lambda_\gm\nu - \nu\|_{\ell^1} \leq |\gm|\,\eps$ for every $\gm \in \Lambda$. Let $R := \lceil 4L\rceil$; Markov's inequality gives $\nu(B_R) \geq 3/4$, where $B_R$ is the ball of radius $R$ around the identity. As $\Lambda$ is infinite, there is $\gm$ with $|\gm| = 2R+1$, and then $\gm B_R \cap B_R = \vide$. Since $\lambda_\gm\nu(\gm B_R) = \nu(B_R) \geq 3/4$, one has $\nu(\gm B_R) \geq 3/4 - (2R+1)\eps$, and $\nu(\gm B_R) + \nu(B_R) \leq 1$ gives $(2R+1)\eps \geq 1/2$. As $2R+1 \leq 8L+3 \leq 11(1+L)$, the claim follows.
\end{proof}

\begin{proof}[Proof of Theorem \ref{e:tA}]
Assume that $H_0 \wr H$ is transport amenable, and let $\eta_n$ and $K$ be as in Definition \ref{e:dtranspam}, for the generating set above (which is allowed, as transport amenability does not depend on the generating set). Telescoping as in the proof of Lemma \ref{e:lspeed} gives $\|\lambda_\gm\eta_n - \eta_n\|_{\ell^1} \to 0$ for every $\gm \in H_0 \wr H$.

\emph{Step 1: test functions.} For $y \in H$, let $L_y(\psi,k) := |\psi(y)|_{H_0}$ be the length of the lamp at $y$, and for $w : H \to [-1,1]$ let
\[
f_w(\psi,k) := \sum_{y \in H} w(y)\, L_y(\psi,k) = \sum_{y \in H} w(y)\, |\psi(y)|_{H_0} ,
\]
a finite sum, as configurations are finitely supported. This function depends on the configuration and not on the position of the lamplighter, and $w$ is a fixed function of the positions $y$ of the lamps. Along an edge of $\Gm$, either the lamplighter moves, which does not change $f_w$, or the lamp under the lamplighter is multiplied by a generator of $H_0$, which changes one term of the sum by at most $|w(k)| \leq 1$. So $\|\nabla f_w\|_{\ell^\infty E} \leq 1$. Let $\mu$ and $\mu'$ be finitely supported probability measures, and let $\tau$ be a transport pattern from $\mu'$ to $\mu$. With the pairings of \S{}\ref{e:stc},
\[
\begin{aligned}
\pgen{\mu \mid f_w} - \pgen{\mu' \mid f_w} & = \pgen{\mu - \mu' \mid f_w} = \pgen{\nabla^*\tau \mid f_w} = \pgen{\tau \mid \nabla f_w} \\
& \leq \|\tau\|_{\ell^1E}\, \|\nabla f_w\|_{\ell^\infty E} \; \leq \; \|\tau\|_{\ell^1E} .
\end{aligned}
\]
Taking the minimum over $\tau$ (Lemma \ref{e:ltcpattern}) and using the symmetry of $\TC$ gives $\TC(\mu,\mu') \geq \pgen{\mu - \mu' \mid f_w}$.

\emph{Step 2: the lamp profile.} Let
\[
\ell_n(y) := \pgen{\eta_n \mid L_y} = \sum_{(\psi,k)} \eta_n(\psi,k)\, |\psi(y)|_{H_0}
\]
be the mean length of the lamp at $y$ under $\eta_n$; as $\eta_n$ is finitely supported, $\ell_n$ is a finitely supported function on $H$. For $b \in S_H$, write $\lambda_b := \lambda_{(\mf{o},b)}$. Left multiplication by $(\mf{o},b)$ moves the lamp at $y$ to the position $by$, \ie $L_y\big((\mf{o},b)\,x\big) = L_{b^{-1}y}(x)$, so that $\pgen{\lambda_b\eta_n \mid L_y} = \pgen{\eta_n \mid L_{b^{-1}y}} = \ell_n(b^{-1}y)$ and
\[
\pgen{\lambda_b\eta_n - \eta_n \mid f_w} = \sum_{y \in H} w(y)\big(\ell_n(b^{-1}y) - \ell_n(y)\big) .
\]
Choosing $w(y) = \mr{sgn}\big(\ell_n(b^{-1}y) - \ell_n(y)\big)$ in Step 1 gives
\[\tag{1}
K \; \geq \; \TC(\lambda_b\eta_n,\eta_n) \; \geq \; \sum_{y \in H}\big|\ell_n(b^{-1}y) - \ell_n(y)\big| \qquad (b \in S_H).
\]
Writing $U_t := \{y \in H \mid \ell_n(y) > t\}$, the layer-cake formula gives
\[\tag{2}
\sum_{b \in S_H}\sum_{y \in H}\big|\ell_n(b^{-1}y) - \ell_n(y)\big| \; = \; \int_0^\infty \sum_{b \in S_H} \big|bU_t \,\Delta\, U_t\big| \dd t .
\]

\emph{Step 3: the case where $H_0$ is infinite.} Left multiplication by a lamp generator $(\delta^a,e_H)$ acts on the lamp at $e_H$ by $\psi(e_H) \mapsto a\psi(e_H)$. So the law $\nu_n$ of $\psi(e_H)$ under $\eta_n$, \ie the push-forward of $\eta_n$ by $(\psi,k) \mapsto \psi(e_H)$, satisfies $\|\lambda_a\nu_n - \nu_n\|_{\ell^1H_0} \leq \|\lambda_{(\delta^a,e_H)}\eta_n - \eta_n\|_{\ell^1} \to 0$ for every $a \in S_{H_0}$, and Lemma \ref{e:lspeed} gives $\ell_n(e_H) = \sum_{g \in H_0} \nu_n(g)\,|g|_{H_0} \to \infty$. For $0 < t < \ell_n(e_H)$, the set $U_t$ is finite and non-empty, hence not invariant under the infinite group $H$; so $bU_t \neq U_t$ for some $b \in S_H$, and then $|bU_t \,\Delta\, U_t| \geq 2$. By (1) and (2), $|S_H|\,K \geq 2\ell_n(e_H)$ for every $n$, a contradiction.

\emph{Step 4: the case where $H_0$ is finite and $H$ is not virtually cyclic.} Fix $y \in H$ and $a \in H_0$. Left multiplication by $(\mf{o},y)(\delta^a,e_H)(\mf{o},y^{-1})$ changes the lamp at $y$ by $\psi(y) \mapsto a\psi(y)$, so that, as in Step 3, the law $\nu_n^y$ of $\psi(y)$ satisfies $\|\lambda_a\nu^y_n - \nu^y_n\|_{\ell^1H_0} \to 0$ for every $a \in H_0$. As $H_0$ is finite, $\nu_n^y$ tends to the uniform measure on $H_0$, and $\ell_n(y) \to c := |H_0|^{-1}\sum_{g \in H_0}|g|_{H_0}$ for every $y \in H$; note that $c>0$ as $H_0 \neq \{1\}$.

Fix $R$. For $n$ large enough, $\ell_n > c/2$ on the ball $B^H_R$ of radius $R$ of $H$, so that $U_t \supseteq B^H_R$ for $0 < t < c/2$. By the isoperimetric inequality of Coulhon \& Saloff-Coste \cite[Th\'eor\`eme 1]{e:CSC}, applied to $U^{-1}$, there is $c' > 0$ such that $\sum_{b \in S_H}|bU \,\Delta\, U| \geq c'|U|/\Phi(2|U|)$ for every finite $U \subset H$, where $\Phi(v) := \min\{r \mid |B^H_r| > v\}$. As $H$ is infinite and not virtually cyclic, $|B^H_r|/r \to \infty$: otherwise, $H$ would be virtually nilpotent by Gromov's theorem \cite{e:Gro81}, in the form of van den Dries \& Wilkie \cite{e:vDW}, which only needs a growth bound along a sequence of radii; but an infinite virtually nilpotent group which is not virtually cyclic contains a copy of $\zz^2$ (a torsion-free nilpotent group of Hirsch length at least $2$ does), and therefore has at least quadratic growth. Hence $v/\Phi(2v) \to \infty$, and
\[
m(R) := \inf\Big\{ \sum_{b \in S_H}|bU \,\Delta\, U| \;\Big|\; U \subset H \text{ finite},\ U \supseteq B^H_R \Big\} \; \xrightarrow[R \to \infty]{} \; \infty .
\]
By (1) and (2), $|S_H|\,K \geq \tfrac{c}{2}\, m(R)$ for every $R$, a contradiction.

So $H_0$ is finite, and $H$ is virtually cyclic.
\end{proof}

\begin{rmk}\label{e:rquant}
Steps 1--3 give a quantitative statement: for every finitely supported probability measure $\eta$ on $H_0 \wr H$ (with $H$ infinite), some base generator $b$ satisfies
\[
\TC(\lambda_b\eta,\eta) \; \geq \; \frac{2}{|S_H|}\, \pgen{\eta \mid L_{e_H}} .
\]
In words, translating the base moves every lamp, whereas along the edges of $\Gm$ a lamp can only be changed, one generator at a time, at the position of the lamplighter. When $H_0$ is finite this bound is harmless, and indeed $H_0 \wr \zz$ is CF. For $F \wr \zz^2$ with $F$ finite and $\eta_m$ the uniform measure on $\{(\psi,k) \mid k \in A_m,\ \mr{supp}\,\psi \subseteq k + A_m\}$, where $A_m = \{-m,\ldots,m\}^2$, the function $f_w$ with $w(y) = \mr{sgn}(y_1)$ gives directly $\TC(\lambda_{b}\eta_m,\eta_m) \geq (4m+1)\big(1 - \tfrac{1}{|F|}\big)$ for the base generator $b = (1,0) \in \zz^2$, when the lamp generating set is $F \setminus \{e_F\}$.
\end{rmk}

\begin{rmk}\label{e:rerror}
In Example 5.1 of \cite{e:moi-trans}, the transport problem is inverted: one displaces $z^{-1}F_n$ to $F_n$, the cost of moving $\gm$ to $\gm'$ being the length of $\gm^{-1}\gm'$. In that inverted picture (with the conventions of \cite{e:moi-trans}), right multiplication moves the lamplighter together with all the lamps, or changes the lamp at $e_H$. An element of $z^{-1}F_n \setminus F_n$ is multiplied on the right by $z$, which brings its lamplighter back into $A_n$, but next to the side from which it left, in a part of $F_n$ which is already fully occupied. The elements of $F_n \setminus z^{-1}F_n$, which have to be reached, have their lamplighter on the opposite side of $A_n$. With these moves, the lamplighter can only be moved relative to the lamps by changing lamps, and reaching the other side requires changing essentially all the lamps of the configuration, not only those on $\Delta_zA_n$, which are the only ones accounted for in Example 5.1. For $H = \zz$ and $H_0$ finite the resulting cost is still bounded, in accordance with Theorem \ref{e:tA}.
\end{rmk}

\begin{rmk}\label{e:rliouville}
Theorem \ref{e:tA} shows that transport amenability does not follow from the Liouville property: the simple random walks on $(\zz/2\zz)\wr\zz^2$ and on $\zz\wr\zz$ are Liouville (see Kaimanovich \& Vershik \cite{e:KV} for the former; for the latter the entropy is sublinear, see \cite{e:Ers-drift}), but these groups are not transport amenable.
% Theorem \ref{e:tA} also applies to iterated wreath products such as $(\zz \wr \zz) \wr \zz$. Note finally that $\zz \wr \zz$ is an automata group: it admits a faithful, finite-state, self-similar action on the ternary rooted tree (Dantas, Santos \& Sidki \cite[Proposition 3.1]{e:DSS}).
\end{rmk}

By Erschler's computation of the F\o lner functions of wreath products \cite[Theorem 1]{e:Ers-isop}, the F\o lner function of $H_0 \wr H$, with $H_0 \neq \{1\}$ and $H$ infinite, is at most exponential exactly when $H_0$ is finite and $H$ is virtually cyclic. Theorem \ref{e:tA} therefore says that, among wreath products with infinite base, transport amenable groups have an at most exponential F\o lner function, as CF groups do (Thom \cite[Remark 6]{e:Thom}; see also \cite[\S{}1]{e:moi-trans}). This suggests the following question.

\begin{ques}\label{e:qcf}
Is every transport amenable group CF? Does transport amenability imply that the F\o lner function is at most exponential?
\end{ques}

We end this section with some remarks on obstructions to transport amenability.

Amenability can be tested on uniform measures on finite sets, whereas transport amenability requires controlling all sequences of measures. For amenability, the two are bridged by the slicing argument of Namioka \cite{e:Namioka}: $\|\lambda_s\mu - \mu\|_{\ell^1}$ is the integral over $t$ of $|sA_t \,\Delta\, A_t|$, where $A_t = [\mu > t]$ are the level sets of $\mu$. This does not work here, at least not in an obvious way: the transport cost is a norm (by Kantorovich--Rubinstein duality) and is therefore only subadditive along this decomposition,
\[
\TC(\lambda_s\mu, \mu) \; \leq \; \int_0^\infty \TC\big(\un_{sA_t}, \un_{A_t}\big) \dd t ,
\]
which is the wrong direction: the transport of one level set can benefit from the others.

On the other hand, duality turns the question into a criterion:
\[
\TC\big(\lambda_s\eta, \eta\big) = \sup \Big\{ \pgen{\lambda_s\eta - \eta \mid f} \; \Big| \; f : G \to \rr \text{ is } 1\text{-Lipschitz for } d_\Gm \Big\} .
\]
An obstruction thus amounts to producing, for every sequence of measures tending to an invariant mean, test functions whose average variation under left translation by a generator is large. As $|f(sy) - f(y)| \leq d_\Gm(sy,y) = |y^{-1}sy|$, such functions must exploit the distortion of conjugation. The functions $f_w$ used in the proof of Theorem \ref{e:tA} are of this kind: conjugating a generator of the base by an element carrying large lamps produces a long element, since the conjugate translates all these lamps. By contrast, first moments such as $\sum_x |x|\,|\lambda_s\mu(x) - \mu(x)|$, which only bound the transport cost from above and may be much larger than it, cannot provide an obstruction.

The argument of Theorem \ref{e:tA} does not seem to apply directly to groups acting on rooted trees, such as the first Grigorchuk group: the analogue of the lamp profile, the mean length of the sections at a given level of the tree, lives on a finite set of vertices and need not decay. We do not know whether such groups are transport amenable.

It should be possible to find groups with a very good isoperimetric profile, for instance with an at most exponential F\o lner function, among automata groups, permutational wreath products or the diagonal products of Brieussel \& Zheng \cite{e:BZ}, which realise a wide range of isoperimetric profiles. Those which are not CF would be the most natural candidates for further transport amenable groups.

%-----------------------------------
\section{Sublinear cocycles and Banach representations}\label{e:ssublin}
%-----------------------------------

Vanishing of $\ell^p$-cohomology is about one particular representation, the left-regular one on $\ell^pG$.
The mechanism behind it is not tied to that choice, and it seems worth recording what it gives in general.
The statement below extends the non-trivial half of Cornulier \& Tessera \cite[Proposition 5.6]{e:CT} and, as a consequence, Thom \cite[Corollary 5]{e:Thom}. Both are stated for CF groups (see the introduction and Question \ref{e:qcf}).

A \textbf{Banach $G$-module} is a Banach space $V$ together with a representation $\pi : G \to GL(V)$ by bounded automorphisms; it is \textbf{uniformly bounded} if $C_\pi := \supp{g \in G} \|\pi(g)\| < \infty$.
Given $\alpha \in V^*$ and $\zeta \in V$, the \textbf{coefficient} $\kappa_{\alpha, \zeta} : G \to \cc$ is $\kappa_{\alpha,\zeta}(x) = \pgen{ \alpha \mid \pi(x)\zeta}$, and $\pi$ is \textbf{strongly mixing} if all its coefficients lie in $c_0G$, \ie if $\kappa_{\alpha,\zeta}(x) \to 0$ as $x \to \infty$ for every $\alpha \in V^*$ and
$\zeta \in V$.
This definition is given in \cite{e:Go-mixing} for isometric representations, but nothing changes for a uniformly bounded $\pi$: the norm $\nnorm{v} := \supp{g\in G}\|\pi(g)v\|$ is equivalent to $\|\cdot\|$ and makes $\pi$ isometric,
without altering the coefficients.
Thom \cite{e:Thom} calls strongly mixing representations $c_0$-representations.

As usual $Z^1(G,\pi) = \{ b : G \to V \mid b(gh) = b(g) + \pi(g)b(h)\}$, the coboundaries are the $\del v : g \mapsto \pi(g)v - v$, and $H^1_{\mr{red}}(G,\pi) = Z^1(G,\pi) / \overline{B^1(G,\pi)}$, the closure being taken for the topology of pointwise convergence.
On a finitely generated group this topology is the one of the norm $\|b\|_S := \maxx{s \in S} \|b(s)\|$.
A cocycle is an \textbf{almost coboundary} if its class in $H^1_{\mr{red}}$ vanishes, and it has \textbf{sublinear growth} if $\|b(g)\|/|g| \to 0$ as $|g| \to \infty$.
Writing a group element as a geodesic word gives $\|b(g)\| \leq C_\pi \|b\|_S |g|$, so that
\[
\eps_b(r) := \supp{|g| \geq r} \frac{\|b(g)\|}{|g|} \qquad (r \geq 1)
\]
is finite and non-increasing, and $b$ is sublinear if and only if $\eps_b(r) \to 0$.

\begin{teo}\label{e:tsublin}
Let $G$ be a finitely generated transport amenable group, let $(V,\pi)$ be a
uniformly bounded Banach $G$-module, and let $b \in Z^1(G,\pi)$ have sublinear
growth. Then $b$ is an almost coboundary.
\end{teo}

\begin{proof}
Let $\sigma(g)v := \pi(g)v + b(g)$ be the associated affine action.
The cocycle identity makes it an action, and since $\sigma(s)v - v = \del v(s) + b(s)$, the cocycle $b$ is an almost coboundary exactly when $\sigma$ admits almost fixed points, \ie when there are $v_k$ with $\maxx{s\in S}\|\sigma(s)v_k - v_k\| \to 0$.

Let $\eta_n$ and $K$ be as in Definition \ref{e:dtranspam}.
Put $v_n := \sum_\gm \eta_n(\gm)b(\gm)$, a finite sum.
Since $\sum_\gm \eta_n(\gm) = 1$ and $\pi(s)b(\gm)+b(s) = b(s\gm)$,
\[
\sigma(s)v_n = \pi(s)v_n + b(s) = \sum_\gm \eta_n(\gm)b(s\gm) = \sum_x (\lambda_s\eta_n)(x)\, b(x),
\]
so that $\sigma(s)v_n - v_n = \pgen{\lambda_s\eta_n - \eta_n \mid b}$.

By Lemma \ref{e:ltcpattern}, pick a cycle-free transport pattern $\tau_n$ from $\eta_n$ to $\lambda_s\eta_n$ with $\|\tau_n\|_{\ell^1E} \leq K$, and let $\gm_i$, running from $x_i$ to $y_i$ and carrying the weight $c_i$, be a flow decomposition of it.
Since $\nabla^*(c\,\tau_\gm) = c(\delta_y - \delta_x)$ for a path $\gm$ from $x$ to $y$, and since $b(y)-b(x) = \pi(x)b(x^{-1}y)$,
\[
\sigma(s)v_n - v_n = \pgen{ \nabla^*\tau_n \mid b}
= \somme{i}{} c_i \big( b(y_i) - b(x_i) \big)
= \somme{i}{} c_i \, \pi(x_i)\, b\big(x_i^{-1}y_i\big) .
\]
Lemma \ref{e:ldisplaced} gives
\[
\somme{i}{} c_i = \tfrac{1}{2}\|\lambda_s\eta_n - \eta_n\|_{\ell^1G} =:
\theta_n(s) \xrightarrow[n \to \infty]{} 0
\qquad \text{and} \qquad
\somme{i}{} c_i |\gm_i| = \|\tau_n\|_{\ell^1E} \leq K .
\]
Split the paths according to the distance $d_i := d_\Gm(x_i,y_i) \leq |\gm_i|$ between their endpoints, at a radius $R \geq 1$. Those with $d_i \leq R$ carry a total weight at most $\theta_n(s)$ and each contributes at most $C_\pi \eps_b(1) R$.
Those with $d_i > R$ contribute at most $C_\pi \eps_b(R) \sum_i c_i d_i \leq C_\pi\eps_b(R) \sum_i c_i|\gm_i| \leq C_\pi\eps_b(R)K$.
Hence
\[
\big\| \sigma(s)v_n - v_n \big\|
\; \leq \; C_\pi \eps_b(1) R\, \theta_n(s) \; + \; C_\pi \eps_b(R) K .
\]
Given $\delta > 0$, choose first $R$ with $C_\pi\eps_b(R)K < \delta/2$, which is possible since $b$ is sublinear, and then $n$ large enough that $C_\pi\eps_b(1)R\,\theta_n(s) < \delta/2$ for the finitely many $s \in S$.
Thus $\maxx{s\in S}\|\sigma(s)v_n - v_n\| \to 0$, \ie $\|b - \del(-v_n)\|_S \to 0$ and $[b] = 0$.
\end{proof}

The two terms of the last estimate correspond to the two halves of the argument: the second is small because the cost is bounded and $b$ is sublinear, the first only because of the ``long displacement'' property (Lemma \ref{e:ldisplaced}), \ie because the mass carried along the short paths tends to $0$.

Since an almost coboundary is sublinear for any uniformly bounded Banach module (the easy half of Cornulier \& Tessera \cite[Proposition 5.6]{e:CT}, which requires nothing of $G$), Theorem \ref{e:tsublin} says that on a transport amenable group the two notions agree.

\begin{cor}\label{e:cthomext}
Let $G$ be a finitely generated transport amenable group, let $V$ be a superreflexive Banach space and let $\pi : G \to GL(V)$ be a uniformly bounded strongly mixing representation.
Then $H^1_{\mr{red}}(G,\pi) = 0$; equivalently, every $1$-cocycle $b : G \to V$ is an almost coboundary.
\end{cor}

\begin{proof}
By Thom \cite[Theorem 1]{e:Thom}, every $1$-cocycle for a uniformly bounded strongly mixing representation on a superreflexive space has sublinear growth.
Apply Theorem \ref{e:tsublin}.
\end{proof}

Since the left-regular representation on $\ell^pG$ is strongly mixing, Corollary \ref{e:cthomext} contains \cite[Theorem 1.3]{e:moi-trans}. The strong mixing hypothesis is used only to produce sublinearity, through \cite[Theorem 1]{e:Thom}, and plays no role in the transport argument.

\end{document}